\documentclass[a4paper,11pt,leqno]{amsart}
\usepackage{latexsym}
\usepackage{amssymb} 
\usepackage{amsmath} 
\usepackage{amscd}
\usepackage{color}
\usepackage{mathtools}

\numberwithin{equation}{section} 

\newtheorem{theorem}{Theorem}[section]
\newtheorem{lemma}[theorem]{Lemma} 
\newtheorem{corollary}[theorem]{Corollary}
\newtheorem{proposition}[theorem]{Proposition} 
 
\newtheorem{remark}[theorem]{Remark}
\newtheorem{example}[theorem]{Example}

\newtheorem{question}[theorem]{Question}

\allowdisplaybreaks[3]

\newcommand{\bi}[2]{\binom{#1}{#2}}

\begin{document}


\title[Computations on the tautological basis]{Computations on the tautological basis of\\the cohomology ring of the Peterson variety}
\author {Yuito Hashimoto}
\address{Department of Applied Mathematics, Okayama University of
Science, 1-1 Ridai-cho, Kita-ku, Okayama, 700-0005, Japan}
\email{yuitohashimoto0811@gmail.com}

\begin{abstract}
It is known that the set of square free monomials on the Chern classes of the tautological line bundles over the Peterson variety forms an additive basis of its cohomology ring. We study the expansion formula for their products. In particular, we give a square free expansion of the products multiplying degree $2$ classes in terms of elementary symmetric polynomials and binomial coefficients.
\end{abstract}

\maketitle

\section{Introduction}
 Let $n$ be a positive integer and $Fl(\mathbb{C}^n)$ the flag variety of $\mathbb{C}^n$ which is the collection of nested sequence of linear subspaces of $\mathbb{C}^n\colon$
\begin{align*}
Fl(\mathbb{C}^n) \ = \ \lbrace V_{\bullet}=(V_0\subset V_1\subset V_2\subset\dotsb\subset V_n=\mathbb{C}^n)\mid \dim_{\mathbb{C}}V_i=i\ (1\le i\le n)\rbrace.
\end{align*} 
 Let $N$ be an $n\times n$ regular nilpotent matrix viewed as a linear map $N\colon\mathbb{C}^n\rightarrow\mathbb{C}^n$. The Peterson variety (of type $A_{n-1}$) is a subvariety of  $Fl(\mathbb{C}^n)$ defined by
\begin{align*}
\mathcal{Y}\coloneqq\lbrace V_{\bullet}\in Fl(\mathbb{C}^n)\mid NV_i\subset V_{i+1}\,(1\le i< n)\rbrace, 
\end{align*}
where $NV_i$ denotes the image of $V_i$ under the map $N\colon\mathbb{C}^n\rightarrow\mathbb{C}^n$. It was introduced by Dale Peterson to study the quantum cohomology ring of $Fl(\mathbb{C}^n)$, and the structure of its cohomology ring has been studied quite well (e.g. \cite{b,j,h,i,g,l,m}). 

Let $x_i\in H^2(\mathcal{Y})\,(1\le i\le n)$ be the Chern class of (the dual of) the tautological line bundle on $\mathcal{Y}$, where $H^*(\mathcal{Y};{\mathbb Q})$ is the singular cohomology ring with ${\mathbb Q}$-coefficients. As is well-known, Peterson variety $\mathcal{Y}$ is an example of regular nilpotent Hessenberg variety. Hence, the result of Harada--Horiguchi--Murai--Precup--Tymoczko (\cite{k}) implies that the set of suquare free monomials 
\begin{align}\label{Petbas1}
\lbrace x_{i_1}\cdots x_{i_d}\mid 1\le i_1<i_2<\cdots<i_d\le n-1\ (1\le d\le n-1)\rbrace
\end{align}
is a basis of $H^{*}(\mathcal{Y})$ over ${\mathbb Q}$. We call \eqref{Petbas1} the \textbf{tautologicai basis} of $H^{*}(\mathcal{Y})$, and we write 
\begin{align*}
x_J\coloneqq x_{i_1}\cdots x_{i_d}\quad \text{for $J=\lbrace i_1<\ldots<i_d\rbrace\subseteq\lbrace 1,\cdots,n-1\rbrace$}.
\end{align*}
For $J,K\subseteq\lbrace 1,\ldots,n-1\rbrace$, we consider the structure constants for the multiplication rule :
\begin{align*}
x_J\cdot x_K \ = \ \sum_Lc_{J,K}^Lx_L,\hspace{20pt}c_{J,K}^L\in{\mathbb Q}.
\end{align*}
The coefficients $c_{J,K}^L$ are defined as rational numbers, but it follows from the computations in this paper that they are all \textit{integers}:
\begin{align*}
 c_{J,K}^L\in{\mathbb Z}.
\end{align*}
We consider the simple case $|J|=1$, and we provide a square free expansion formula by using elementary symmetric polynomials in $x_1,\ldots,x_{n-1}$ and  binomial coefficients (Theorem~\ref{xsqu} and Theorem~\ref{thm12}). 
We address that our formula gives an expansion by the tautological basis, but it does not necessarily give a formula for the coefficients $c_{J,K}^L$ directly (see Remark~\ref{remarkcoe} and Proposition~\ref{cor5.2}).
We list some open questions in Section~\ref{opque} including geometric or combinatorial meaning for these coefficients.

\vspace{10pt}

\noindent
\textbf{Acknowledgments.} The author is greatful to Hiraku Abe for valuable discussions and suggestions.
He also thanks the anonymous referee for reading the manuscript carefully and giving suggestions which improved the accuracy of this paper.

\section{Preliminaries}
Flag variety is defined as the set of nested sequence of linear subspaces in $\mathbb{C}^n$:
\begin{align*}
Fl(\mathbb{C}^n)=\lbrace V_{\bullet}=(V_0\subset V_1\subset V_2\subset\cdots\subset V_n=\mathbb{C}^n)\mid \dim_{\mathbb{C}}V_i=i\,(1\le i\le n)\rbrace.
\end{align*}
For $1\le i\le n$, the $i$-th tautological vector bundle (of rank $i$) over $Fl(\mathbb{C}^n)$ is defined as
\begin{align*}
U_i\coloneqq \lbrace (V_{\bullet},v)\in Fl(\mathbb{C}^n)\times\mathbb{C}^n\mid V_i\ni v\rbrace\hspace{15pt}(1\le i\le n).
\end{align*}
Then we have the following sequence of inclusions of tautological vector bundles:
\begin{align*}
0=U_0\subset U_1\subset U_2\subset\cdots\subset U_{n-1}\subset U_n=Fl(\mathbb{C}^n)\times\mathbb{C}^n .
\end{align*}
In this sequence, each successive quotient $U_i/U_{i-1}$ is a line bundle over $Fl(\mathbb{C}^n)$, and it is called a \textbf{tautological line bundle}. It has the associated first Chern class
\begin{align}\label{tau}
\tau_i\coloneqq -c_1(U_i/U_{i-1})\in H^{2}(Fl(\mathbb{C}^n))\hspace{15pt}(1\le i\le n).
\end{align}

Let us now describe the cohomology ring of the Peterson variety.
Let $N$ be the $n\times n$ nilpotent matrix in Jordan canonical form consisting of a single Jordan block:
\begin{align*}
N\coloneqq
\begin{pmatrix} 
  0 & 1 & 0 &\dots  & 0 \\
  0 & 0 & 1 &\dots  & 0 \\
  \vdots & \vdots & \vdots & \ddots & \vdots \\
  0 & 0 & 0 &\dots  & 1\\
  0 & 0 & 0 &\dots  & 0
\end{pmatrix} .
\end{align*}
Regarding $N$ as a linear map $N\colon \mathbb{C}^n\rightarrow\mathbb{C}^n$, the Peterson variety is defined as a subvariety of $Fl(\mathbb{C}^n)$ as follows:
\begin{align*}
\mathcal{Y}\coloneqq\lbrace V_{\bullet}\in Fl(\mathbb{C}^n)\mid NV_i\subset V_{i+1} \hspace{5pt}(1\le i\le n-1)\rbrace,
\end{align*}
where $NV_i$ denotes the image of $V_i$ under the linear map $N\colon \mathbb{C}^n\rightarrow\mathbb{C}^n$.

In this paper, $H^*(-;{\mathbb Q})$ denotes the singular cohomology ring with ${\mathbb Q}$-coefficients.
The inclusion $\pi\colon\mathcal{Y}\hookrightarrow Fl(\mathbb{C}^n)$ induces a ring homomorphism 
\begin{align*}
\pi^{*}\colon H^{*}(Fl(\mathbb{C}^n))\rightarrow H^{*}(\mathcal{Y}).
\end{align*}
According to Fukukawa--Harada--Masuda (\cite{a}), we set 
\begin{align}\label{vpidef}
p_k\coloneqq \sum_{i=1}^{k}\pi^{*}(\tau_i) \in H^{2}(\mathcal{Y}) \hspace{10pt}(1\le k\le n-1), 
\end{align}
where $\tau_i= -c_1(U_i/U_{i-1})\in H^{2}(Fl(\mathbb{C}^n))$ is the Chern class defined in \eqref{tau}. Then we have the following presentation of the cohomology ring of $\mathcal{Y}$.

\begin{proposition}\label{Pet}{\rm (\cite[Corollary 3.4]{a})}
There is a ring isomorphism 
\begin{align*}
\mathbb{Q}\lbrack \varpi_1,\ldots,\varpi_{n-1}\rbrack/J_{\varpi} \stackrel{\cong}{\rightarrow} H^{*}(\mathcal{Y})
\end{align*}
which sends $($the quotient image of\hspace{3pt}$)$ $\varpi_k$ to $p_k\in H^{2}(\mathcal{Y})$,
where $J_{\varpi}$ is the ideal generated by 
\begin{align*}
\varpi_k(-\varpi_{k-1}+2\varpi_k-\varpi_{k+1})\hspace{10pt}(1\le k\le n-1), 
\end{align*}
where we set $\varpi_0=\varpi_n=0$.
\end{proposition}

\vspace{10pt}

Due to this result, we identify (the quotient image of) $\varpi_k$ and the cohomology class $p_k$, and we regard
\begin{align*}
 \varpi_k \in H^{2}(\mathcal{Y})\ (1\le k\le n-1)
\end{align*}
in the rest of this paper. This does not cause problems since we always compute in the cohomology ring $H^{*}(\mathcal{Y})$ when we deal with these classes.

It is known that the square free monomials in $\varpi_1,\ldots,\varpi_{n-1}$ forms a basis of $H^{*}(\mathcal{Y})$ over ${\mathbb Q}$ (\cite{b,i,g}).
This means that an arbitrary monomial in $\varpi_1,\ldots,\varpi_{n-1}$ must be written as a linear combination of the ones which are square free. Among results in the literature, the next claim is useful for the computations in this paper.

\begin{proposition}{\rm (\cite[Lemma 5.1]{b})}\label{varpisqu}
For $1\le a\le i\le b\le n-1$, we have 
\begin{align}\label{varpiiaa+1}
\varpi_i(\varpi_a\varpi_{a+1}\cdots \varpi_b)=\frac{b-i+1}{b-a+2}\,\varpi_{a-1}\varpi_a\cdots \varpi_b+\frac{i-a+1}{b-a+2}\,\varpi_a\varpi_{a+1}\cdots \varpi_{b+1}
\end{align}
in $H^{*}(\mathcal{Y})$, where we set $\varpi_0=\varpi_n=0$.
\end{proposition}

\vspace{10pt}

\begin{remark}\label{remarkpi}
{\rm When $i=a-1$, the coefficients of the first and the second summands in the right hand side become $\frac{b-i+1}{b-a+2}=1$ and $\frac{i-a+1}{b-a+2}=0$, respectively. Hence, \eqref{varpiiaa+1} also holds when $i=a-1$. Similarly, it holds when $i=b+1$.}
\end{remark}

\vspace{10pt}

We also consider the cohomology class
\begin{align}\label{xdef}
y_i\coloneqq \pi^{*}(\tau_i)\in H^{2}(\mathcal{Y}) \hspace{10pt}(1\le i\le n).
\end{align}
This is simply the restriction of the Chern class defined in \eqref{varpisqu}.
Recalling the definition of $p_k$ given in $(\ref{vpidef})$, we have
\begin{align*}
p_k=y_1+\cdots+y_k\hspace{10pt}(1\le k\le n).
\end{align*}
Now the ring presentation of $H^{*}(\mathcal{Y})$ (Proposition~\ref{Pet}) given by Fukukawa--Harada--Masuda (\cite{a}) can be expressed in terms of $y_1,\ldots,y_n$ as follows.

\begin{corollary}\label{petx}
There is a ring isomorphism 
\begin{align*}
\mathbb{Q}\lbrack x_1,\ldots,x_n\rbrack/J_x 
\stackrel{\cong}{\rightarrow} H^{*}(\mathcal{Y})
\end{align*}
which sends $($the quotient image of\hspace{3pt}$)$ $x_i$ to $y_i\in H^{2}(\mathcal{Y})$ given in \eqref{xdef},
where $J_x$ is the ideal generated by 
\begin{align*}
(x_1+\cdots+x_k)(x_k-x_{k+1})\ (1\le k\le n-1),\quad x_1+\cdots+x_n.
\end{align*}
\end{corollary}

\vspace{10pt}

Similar to Proposition~\ref{Pet}, this makes us to identify (the quotient image of) $x_i$ and the cohomology class $y_i$, and we regard
\begin{align*}
 x_i \in H^{2}(\mathcal{Y})\ (1\le k\le n)
\end{align*}
in the rest of this paper. With this understanding, the above equality $p_k=y_1+\cdots+y_k$ takes of the form
\begin{align}\label{xvpiey}
\varpi_k=x_1+\cdots+x_k\hspace{10pt}(1\le k\le n).
\end{align}
In \cite{c}, Abe--Zeng generalized this relation to higher degree cases by using elementary symmetric polynomials as follows.

\begin{proposition}{\rm (\cite[Proposition 4.6]{c})}\label{piele}
For  $1\le a\le b\le n-1$, the following holds in $H^{*}(\mathcal{Y})$$:$
\begin{align*}
\underbrace{\varpi_a\varpi_{a+1}\cdots\varpi_b}_{\text{$k$}}=k!e_k(x_1,x_2,\ldots,x_b),
\end{align*}
where $k=b-a+1$.
\end{proposition}

\vspace{10pt}

Noticing that the Peterson variety $\mathcal{Y}$ is an example of regular nilpotent Hessenberg variety, we can use the following result of Harada--Horiguchi--Murai--Precup--Tymoczko.

\begin{proposition}\label{xbase}
$($\cite[Corollary7.3]{k}$)$
The following set of square free monomials in $x_1,\ldots,x_{n-1}$ forms a basis of $H^{*}(\mathcal{Y})$ over ${\mathbb Q}$$:$
\begin{align}\label{Petbas}
\lbrace x_{i_1}\cdots x_{i_m}\mid 1\le i_1<i_2<\cdots<i_m\le n-1\ (1\le m\le n-1)\rbrace.
\end{align}
\end{proposition}

\vspace{5pt}

We note that $x_n$ does not appear in \eqref{Petbas}. This can be understood from the linear relation $x_1+\cdots+x_n(=\varpi_n)=0$ in $H^{2}(\mathcal{Y})$ (see Corollary~\ref{petx}).
Since the basis in \eqref{Petbas} uses Chern classes of tautological line bundles, we call this basis the \textbf{tautological basis} of $H^{*}(\mathcal{Y})$.

\begin{example}
{\rm
Let $n=4$. Then the tautological basis of $H^{*}(\mathcal{Y})$ consists of 
\begin{align*}
 1,\ \ x_1,\ x_2,\ x_3,\ \ 
 x_1x_2,\ x_1x_3,\ x_2x_3,\ \ x_1x_2x_3.
\end{align*}
}
\end{example}

\vspace{10pt}

Proposition~\ref{xbase} shows that an arbitrary monomial in $x_1,\ldots,x_{n-1}$ in $H^{*}(\mathcal{Y})$ can be written as a linear combination of square free monomials in $x_1,\ldots,x_{n-1}$ with rational coefficients.
However, it does not provide an explicit formula for these coefficients.
In this paper, we give an explicit expansion formula for a monomial which containing a single square.

\vspace{10pt}

We end this section by recording an interesting property of an alternating sum of products of binomial coefficients.
\begin{lemma}\label{nikou}
For $d>b\ge0$, we have
\begin{align*}
\sum_{i=b}^{d-1}(-1)^{i-b}\bi{i}{b}\bi{d}{i+1} \ = \ 1.
\end{align*}
\end{lemma}

\vspace{10pt}

\begin{example}
{\rm
Let $d=6$ and $b=2$. Then we have 
\begin{align*}
\sum_{i=2}^5(-1)^{i-2}\bi{i}{2}\bi{6}{i+1} 
\ &= \ \bi{2}{2}\bi{6}{3}-\bi{3}{2}\bi{6}{4}+\bi{4}{2}\bi{6}{5}-\bi{5}{2}\bi{6}{6}\ = \ 1.
\end{align*}
}
\end{example}

\vspace{10pt}

In Appendix, we give a proof of Lemma~\ref{nikou} for the reader since the author could not find a suitable reference.

\vspace{20pt}

\section{Consecutive case}\label{connected case}
In this section, we study the expansion of the monomial
\begin{align}\label{conxeq}
x_c(x_ax_{a+1}\cdots x_b) \qquad \text{for \quad $1\le a\le c\le b\le n-1$}
\end{align}
having consecutive indexes and a unique square (in $x_c$) by the tautological basis of $H^*(\mathcal{Y})$. The cases for general monomials will be treated in the next section.

For $1\le d\le n$ and $J=\lbrace i_1<\cdots<i_d\rbrace\subseteq\lbrace1,\ldots,n\rbrace$, we set
\begin{align}\label{xaxb1sq}
x_J\coloneqq x_{i_1}x_{i_2}\cdots x_{i_d}.
\end{align}
We also write $x_J=x_{ i_1,\ldots,i_d}$ for simplicity.
As a convention, we also set $x_{ i_1,\ldots,i_d}=1$ when $d=0$.
As we saw in the previous section, the cohomology classes $x_1,\ldots,x_n\in H^2(\mathcal{Y})$ are linearly dependent by the relation $x_1+\cdots+x_n=0$, but for the moment we do not exclude $x_n$. Namely, we first study an expansion formula for \eqref{xaxb1sq} by square free monomials in $x_1,\ldots,x_n$. \\

We prepare two results which we use in this paper. 

\begin{lemma}\label{s-k+1ls}
Let $k\ge 1$ and $1\le m\le n-1$. If $m-k+1\le l\le m$, the following holds in $H^{*}(\mathcal{Y})$$:$
\begin{align*}
x_le_k(x_1,\ldots,x_m) \ = \ e_k(x_1,\ldots,x_m)x_{m+1}.
\end{align*}
\end{lemma}

\begin{proof}
When $k>m$, both sides are $0$ so that the claim holds.
Hence, we may assume that $k\le m$ below.
By the assumption $m-k+1\le l\le m$, we can write $l=m-k+i$ for some $1\le i\le k$.
Then we have
\begin{align*}
x_le_k(x_1,\ldots,x_m) \ = \ x_{m-k+i}e_k(x_1,\ldots,x_m).
\end{align*}
By $x_{m-k+i}=\varpi_{m-k+i}-\varpi_{m-k+i-1}$ (see \eqref{xvpiey}) and Proposition~\ref{piele}, we obtain
\begin{align*}
x_le_k(x_1,\ldots,x_m) \ &= \ (\varpi_{m-k+i}-\varpi_{m-k+i-1})\frac{1}{k!}\varpi_{m-k+1}\cdots\varpi_m\\
\ &= \ \frac{1}{k!}\Big(\varpi_{m-k+i}(\varpi_{m-k+1}\cdots\varpi_m)-\varpi_{m-k+i-1}(\varpi_{m-k+1}\cdots\varpi_m)\Big).
\end{align*}
Applying Proposition~\ref{varpisqu} (with Remark~\ref{remarkpi} in mind) to the right hand side, we obtain
\begin{align*}
x_le_k(x_1,\ldots,x_m) 
\ &= \ \frac{1}{k!}\Big(\frac{k-i+1}{k+1}\varpi_{m-k}\cdots\varpi_m+\frac{i}{k+1}\varpi_{m-k+1}\cdots\varpi_{m+1}\\
&\hspace{60pt}-\frac{k-i+2}{k+1}\varpi_{m-k}\cdots\varpi_m-\frac{i-1}{k+1}\varpi_{m-k+1}\cdots\varpi_{m+1}\Big)\\
\ &= \ \frac{1}{(k+1)!}(-\varpi_{m-k}\cdots\varpi_m+\varpi_{m-k+1}\cdots\varpi_{m+1})\\[-13pt]\\
\ &= \ -e_{k+1}(x_1,\ldots,x_m)+e_{k+1}(x_1,\ldots,x_{m+1})\hspace{15pt}(\text{by Proposition~\ref{piele}})\\[-13pt]\\
\ &= \ e_k(x_1,\ldots,x_m)x_{m+1}.
\end{align*}
Thus, the claim holds.
\end{proof}

\vspace{10pt}
For the elementary symmetric polynomial of degree $0$, 
we set
\begin{align}\label{convention of e}
e_0(x_1,\ldots,x_m) \ = \ 1 \qquad (1\le m\le n).
\end{align}
We also set
\begin{align}\label{convention of 1.5}
\text{When $m=0$ \ \ : \ \ $e_0(x_1,\ldots,x_m)=1$ and $e_i(x_1,\ldots,x_m)=0$ $(i\ge1)$}
\end{align}
as convention.
Then we have the well-known identity
\begin{align}\label{convention of e 2}
 e_i(x_1,\ldots,x_{t-1})x_t + e_{i+1}(x_1,\ldots,x_{t-1}) = e_{i+1}(x_1,\ldots,x_t) 
\end{align}
for $i\ge 0$ and $1\le t\le n$ as is well-known. The only non-trivial cases are the case $i=0$ and the case $i\ge 1$ with $t=1$.

\begin{lemma}\label{lem5}
For $i\ge 0$ and $1\le t\le n-1$, the following holds in $H^{*}(\mathcal{Y})$$:$
\begin{align*}
e_i(x_1,\ldots,x_{t-1})x_t^2 \ = \ -e_{i+1}(x_1,\ldots,x_{t-1})x_t+e_{i+1}(x_1,\ldots,x_t)x_{t+1}.
\end{align*}
\end{lemma}

\begin{proof}
Let us rather prove the following equivalent equality:
\begin{align*}
e_i(x_1,\ldots,x_t)x_{t+1}^2 \ +\ e_{i+1}(x_1,\ldots,x_t)x_{t+1}
\ = \ e_{i+1}(x_1,\ldots,x_{t+1})x_{t+2}.
\end{align*}
The left hand side can be computed as 
\begin{align*}
&e_i(x_1,\ldots,x_{t-1})x_t^2 \ +\ e_{i+1}(x_1,\ldots,x_{t-1})x_t \\
&\hspace{20pt}\ = \ \Big( e_i(x_1,\ldots,x_{t-1})x_t + e_{i+1}(x_1,\ldots,x_{t-1}) \Big)x_t \\
&\hspace{20pt}\ = \ \Big( e_{i+1}(x_1,\ldots,x_t) \Big)x_t \qquad \text{(by \eqref{convention of e 2})}\\
&\hspace{20pt}\ = \ e_{i+1}(x_1,\ldots,x_t)x_{t+1} ,
\end{align*}
where we used Lemma~\ref{s-k+1ls}.
Hence, the claim follows.
\end{proof}

\vspace{10pt}

With these preparations, we now study the following special case of \eqref{conxeq}.

\begin{proposition}\label{lemend}
For $1\le a\le b\le n-1$, the following holds in $H^{*}(\mathcal{Y})$$:$
\begin{align*}
\begin{split}
x_b(x_ax_{a+1}\cdots x_b) 
\ &= \ (-1)^{b-a-1}e_{b-a+1}(x_1,\ldots,x_{b-1})x_{b}\\
&\hspace{40pt}+(-1)^{b-a}e_{b-a+1}(x_1,\ldots,x_{b-1})x_{b+1}\ +\ x_{a,\ldots,b+1}.
\end{split}
\end{align*}
\end{proposition}
\begin{proof}
We prove the claim by induction on the length $l\coloneqq b-a+1(\ge1)$ of the monomial $x_a\cdots x_b$ on the left hand side.
For the case $l=1$, recall that we have 
\begin{align*}
(x_1+\cdots+x_a)(x_a-x_{a+1})=0
\end{align*}
for $1\le a\le n-1$ from Corollary~\ref{petx}.
Expanding the product on the left hand side, we obtain
\begin{align}\label{m+1 a2}
\begin{split}
x_a^2 \ &= \ -(x_1+\cdots+x_{a-1})x_a+(x_1+\cdots+x_{a})x_{a+1}\\
\ &= \ -e_1(x_1,\ldots,x_{a-1})x_a+e_1(x_1,\ldots,x_{a-1})x_{a+1}+x_ax_{a+1}.
\end{split}
\end{align}
This gives the claim for $l=1$.

We next assume that the claim holds for some length $1\le l<n-1$, and we prove the claim for the length $l+1$.
For that purpose, take $1\le a\le b\le n-1$ which satisfy $b-a+1=l+1$.
Then we have $b-(a+1)+1=l$ so that the inductive hypothesis implies that 
\begin{align*}
\begin{split}
x_b(x_{a+1}\cdots x_b) 
\ &= \ (-1)^{b-a}e_{b-a}(x_1,\ldots,x_{b-1})x_{b}\\
&\hspace{40pt}+(-1)^{b-a-1}e_{b-a}(x_1,\ldots,x_{b-1})x_{b+1}\ +\ x_{a+1,\ldots,b+1}.
\end{split}
\end{align*}
Multiplying $x_a$ on both sides, we obtain
\begin{equation}\label{m+1}
\begin{split}
x_{b}(x_a \cdots x_{b}) \ &= \ (-1)^{b-a}x_ae_{b-a}(x_1,\ldots,x_{b-1})x_{b}\\
&\hspace{40pt}+(-1)^{b-a-1}x_ae_{b-a}(x_1,\ldots,x_{b-1})x_{b+1}\ +\ x_{a,\ldots,b+1} .
\end{split}
\end{equation}
Since the third summand on the right hand side is already square free, it is enough to resolve the squares in the first and the second summands.
Notice that both of these summands have the common part $x_ae_{b-a}(x_1,\ldots,x_{b-1})$ with $b-a\ge1$ and $b-1(\ge a)\ge1$. For this part, we have
\begin{align*}
x_ae_{b-a}(x_1,\ldots,x_{b-1}) \ &= \ e_{b-a}(x_1,\ldots,x_{b-1})x_b
\end{align*}
by Lemma~\ref{s-k+1ls}.
Applying this to the first and the second summands in $(\ref{m+1})$, we obtain
\begin{align}\label{m+12}
\begin{split}
x_b(x_a\cdots x_{b}) \ &= \ (-1)^{b-a}e_{b-a}(x_1,\ldots,x_{b-1})x_b^2 \\
&\hspace{40pt}+(-1)^{b-a-1}e_{b-a}(x_1,\ldots,x_{b-1})x_bx_{b+1}\ +\ x_{a,\ldots,b+1}.
\end{split}
\end{align}
In this expression, the only term which has square is $e_{b-a}(x_1,\ldots,x_{b-1})x_b^2$ in the first summand. So let us express it by Lemma~\ref{lem5}:
\begin{align*}
e_{b-a}(x_1,\ldots,x_{b-1})x_b^2 \ = \ -e_{b-a+1}(x_1,\ldots,x_{b-1})x_b+e_{b-a+1}(x_1,\ldots,x_b)x_{b+1}.
\end{align*}
Applying this result to $(\ref{m+12})$, we obtain
\begin{align*}
&x_{b}(x_a \cdots x_{b}) \\
\ &\hspace{20pt}= \ (-1)^{b-a}\Big(-e_{b-a+1}(x_1,\ldots,x_{b-1})x_b+e_{b-a+1}(x_1,\ldots,x_b)x_{b+1}\Big)\\
&\hspace{60pt}+(-1)^{b-a-1}e_{b-a}(x_1,\ldots,x_{b-1})x_bx_{b+1}\ +\ x_{a,\ldots,b+1} \\[-13pt]\\
\ &\hspace{20pt}= \ (-1)^{b-a}\Big(-e_{b-a+1}(x_1,\ldots,x_{b-1})x_b\\
&\hspace{90pt}+e_{b-a+1}(x_1,\ldots,x_{b-1})x_{b+1}+e_{b-a}(x_1,\ldots,x_{b-1})x_bx_{b+1}\Big)\\
&\hspace{60pt}+(-1)^{b-a-1}e_{b-a}(x_1,\ldots,x_{b-1})x_bx_{b+1}\ +\ x_{a,\ldots,b+1} \quad\qquad \text{(by \eqref{convention of e 2})}\\[-13pt]\\
\ &\hspace{20pt}= \ (-1)^{b-a-1}e_{b-a+1}(x_1,\ldots,x_{b-1})x_{b}\\
&\hspace{60pt}+(-1)^{b-a}e_{b-a+1}(x_1,\ldots,x_{b-1})x_{b+1}\ +\ x_{a,\ldots,b+1}, 
\end{align*}
where the last equality follows since the third and the fourth summands cancelled out.
This proves the claim for the length $l+1$, and the proof completes by induction.
\end{proof}

\vspace{10pt}

\begin{example}
{\rm
Let $n=7$. For $a=3$ and $b=5$, Proposition~$\mathrm{\ref{lemend}}$ shows that
\begin{align*}
x_5\cdot (x_3x_4x_5) \ = \ -e_3(x_1,\ldots,x_4)x_5+e_3(x_1,\ldots,x_4)x_6+x_{3,4,5,6}.
\end{align*}
For the elementary symmetric polynomial appearing here, we have
\begin{align*}
e_3(x_1,\ldots,x_4) \ = \ x_{1,2,3}+x_{1,2,4}+x_{1,3,4}+x_{2,3,4}.
\end{align*}
So we obtain that
\begin{align*}
x_5\cdot (x_3x_4x_5) \ &= \ -x_{1,2,3,5}-x_{1,2,4,5}-x_{1,3,4,5}-x_{2,3,4,5}\\
&\hspace{30pt}+x_{1,2,3,6}+x_{1,2,4,6}+x_{1,3,4,6}+x_{2,3,4,6}+x_{3,4,5,6}.
\end{align*} 
}
\end{example}

\vspace{20pt}

We now consider the expansion of the product $x_c(x_a\cdots x_b)$ for $a\le c\le b$ by square free monomials (see \eqref{conxeq}). Before giving the detail formula, let us show an example which helps the reader to understand it.
For example, if $n=7$, the formula we will see in the next theorem shows that 
\begin{equation}\label{eg10}
\begin{split}
x_4(x_3x_4x_5) \ &= \ e_{2}(x_1,\ldots,x_3)x_{4,5}+2e_{3}(x_1,\ldots,x_4)x_{5}\\
&\hspace{15pt}-e_{2}(x_1,\ldots,x_3)x_{5,6}-2e_{3}(x_1,\ldots,x_4)x_{6}+x_{3,4,5,6} .
\end{split}
\end{equation}
One can see that the elementary symmetric polynomials appear as pairs and that the coefficients for each of these pairs are the same up to sign.
Also, the last summand $x_{3,4,5,6}$ does appear as in Proposition~\ref{lemend}. We now describe the formal statement.

\begin{theorem}\label{xsqu}
For $1\le a\le c\le b\le n-1$, the following holds in $H^{*}(\mathcal{Y})$$:$
\begin{equation}\label{xsqueq}
\begin{split}
x_c(x_ax_{a+1}\cdots x_b) \ &= \ (-1)^{c-a-1}\sum_{i=c}^{b}\bi{i-a}{c-a}e_{i-a+1}(x_1,\ldots,x_{i-1})x_{i,\ldots,b}\\
&\hspace{19pt}+(-1)^{c-a}\sum_{i=c}^{b}\bi{i-a}{c-a}e_{i-a+1}(x_1,\ldots,x_{i-1})x_{i+1,\ldots,b+1}\ + \ x_{a,\ldots,b+1}.
\end{split}
\end{equation}
\end{theorem}

\vspace{5pt}

\begin{remark}
{\rm When $b\le n-2$, the class $x_n$ does not appear on the right hand side (e.g.,\ see the explicit form for $x_4(x_3x_4x_5)$ above) so that the right hand side is a linear combination of tautological basis of $H^{*}(\mathcal{Y})$.
When $b=n-1$, it does appear on the right hand side. The expansion in terms of tautological basis in this case will be treated in Theorem~\ref{b=n-1}.
}
\end{remark}

\vspace{5pt}

\begin{proof}[Proof of Theorem~$\ref{xsqu}$]
We prove the claim by induction on the length $l\coloneqq b-a+1(\ge1)$ of $x_a\cdots x_b$.
When $l=1$, we have $a=b=c$, and hence the claim follows from the case of $b=a$ of Proposition~$\ref{lemend}$ (or simply from \eqref{m+1 a2}).

We next assume that the claim holds for some length $1\le l<n-1$, and we prove the claim for the length $l+1$.
For that purpose, take $1\le a\le c\le b\le n-1$ which satisfies $b-a+1=l+1$.
Since we have $(b-1)-a+1=l$, the inductive hypothesis implies that if $c\le b-1$ we have
\begin{align*}
x_c(x_a\cdots x_{b-1}) \ &= \ (-1)^{c-a-1}\sum_{i=c}^{b-1}\bi{i-a}{c-a}e_{i-a+1}(x_1,\ldots,x_{i-1})x_{i,\ldots,b-1}\\
&\hspace{50pt}+(-1)^{c-a}\sum_{i=c}^{b-1}\bi{i-a}{c-a}e_{i-a+1}(x_1,\ldots,x_{i-1})x_{i+1,\ldots,b}\ +\ x_{a,\ldots,b} .
\end{align*}
Note that this equality  still holds even if $b=c$ since both sides are equal to $x_{a,\ldots,b}$ in that case. In particular, the equality holds without conditions on the range of $c$.
Multiplying $x_{b}$ on both sides of this equality, we obtain 
\begin{align}\label{cm+1}
\begin{split}
x_c(x_a\cdots x_{b}) \ &= \ (-1)^{c-a-1}\sum_{i=c}^{b-1}\bi{i-a}{c-a}e_{i-a+1}(x_1,\ldots,x_{i-1})x_{i,\ldots,b}\\
&\hspace{30pt}+(-1)^{c-a}\sum_{i=c}^{b-1}\bi{i-a}{c-a}e_{i-a+1}(x_1,\ldots,x_{i-1})x_{i+1,\ldots,b-1}x_{b}^2\\
&\hspace{30pt}+x_{a,\ldots,b-1}x_{b}^2.
\end{split}
\end{align}
Noticing that the first summand is already square free, it suffices to resolve the squares in the second and the third summands. 
The former can be computed by Proposition~\ref{lemend} as 
\begin{align*}
&e_{i-a+1}(x_1,\ldots,x_{i-1})x_{i+1,\ldots,b-1}x_{b}^2\\
\ &\hspace{20pt}= \  e_{i-a+1}(x_1,\ldots,x_{i-1})\Big((-1)^{b-i-2}e_{b-i}(x_1,\ldots,x_{b-1})x_{b}\\
&\hspace{170pt}+(-1)^{b-i-1}e_{b-i}(x_1,\ldots,x_{b-1})x_{b+1}\ +\ x_{i+1,\ldots,b+1}\Big).
\end{align*}
Applying Proposition~\ref{piele} to the three elementary symmetric polynomials on the right hand side, this computation can be proceeded as
\begin{align*}
&e_{i-a+1}(x_1,\ldots,x_{i-1})x_{i+1,\ldots,{b-1}}x_{b}^2\\
\ &\hspace{15pt}= \  \frac{1}{(i-a+1)!}\varpi_{a-1}\cdots \varpi_{i-1}\Big{(}(-1)^{b-i-2}\frac{1}{(b-i)!}\varpi_i\cdots \varpi_{b-1}x_{b}\\
&\hspace{165pt}+(-1)^{b-i-1}\frac{1}{(b-i)!}\varpi_i\cdots \varpi_{b-1}x_{b+1}\ +\ x_{i+1,\ldots,b+1}\Big{)} \\
\ &\hspace{15pt}= \  (-1)^{b-i-2}\frac{(b-a+1)!}{(i-a+1)!(b-i)!}\Big( \frac{1}{(b-a+1)!}\,\varpi_{a-1}\cdots \varpi_{b-1} \Big) x_{b} \\
&\hspace{60pt}+(-1)^{b-i-1}\frac{(b-a+1)!}{(i-a+1)!(b-i)!} \Big( \frac{1}{(b-a+1)!}\,\varpi_{a-1}\cdots \varpi_{b-1} \Big) x_{b+1} \\
&\hspace{60pt}+\Big( \frac{1}{(i-a+1)!}\varpi_{a-1}\cdots \varpi_{i-1} \Big) x_{i+1,\ldots,b+1}
\end{align*}
\begin{align*}
&\hspace{15pt}= \  (-1)^{b-i-2}\bi{b-a+1}{i-a+1}e_{b-a+1}(x_1,\ldots,x_{b-1})x_{b}\\
&\hspace{60pt}+(-1)^{b-i-1}\bi{b-a+1}{i-a+1}e_{b-a+1}(x_1,\ldots,x_{b-1})x_{b+1}\\[-13pt]\\
&\hspace{60pt}+e_{i-a+1}(x_1,\ldots,x_{i-1})x_{i+1,\ldots,b+1} \hspace{30pt}\text{(by Proposition~\ref{piele} again)}.
\end{align*}
For the third summand in \eqref{cm+1}, we have
\begin{equation*}
\begin{split}
x_{a,\ldots,b-1}x_{b}^2 \ &= \ (-1)^{b-a-1}e_{b-a+1}(x_1,\ldots,x_{b-1})x_{b}\\
&\hspace{40pt}+(-1)^{b-a}e_{b-a+1}(x_1,\ldots,x_{b-1})x_{b+1}\ +\ x_{a,\ldots,b+1}
\end{split}
\end{equation*}
by Proposition~\ref{lemend}.
Applying these two equalities to the second and the third summands in $(\ref{cm+1})$, it follows that 
\begin{align*}
&x_c(x_a\cdots x_{b}) \\
&\hspace{20pt}= \ (-1)^{c-a-1}\sum_{i=c}^{b-1}\bi{i-a}{c-a}e_{i-a+1}(x_1,\ldots,x_{i-1})x_{i,\ldots,b}\nonumber\\
&\hspace{60pt}+(-1)^{c-a}\sum_{i=c}^{b-1}\bi{i-a}{c-a}\Big\{(-1)^{b-i-2}\bi{b-a+1}{i-a+1}e_{b-a+1}(x_1,\ldots,x_{b-1})x_{b}\nonumber\\
&\hspace{175pt}+(-1)^{b-i-1}\bi{b-a+1}{i-a+1}e_{b-a+1}(x_1,\ldots,x_{b-1})x_{b+1}\nonumber\\
&\hspace{175pt}+e_{i-a+1}(x_1,\ldots,x_{i-1})x_{i+1,\ldots,b+1}\Big\}\nonumber\\
&\hspace{60pt}+(-1)^{b-a-1}e_{b-a+1}(x_1,\ldots,x_{b-1})x_{b}\nonumber\\[-13pt]\\
&\hspace{60pt}+(-1)^{b-a}e_{b-a+1}(x_1,\ldots,x_{b-1})x_{b+1}\ +\ x_{a,\ldots,b+1} .
\end{align*}
By combining terms proportional to each elementary symmetric function, we obtain
\begin{align*}
&x_c(x_a\cdots x_{b}) \\
&\qquad= \ (-1)^{c-a-1}\sum_{i=c}^{b-1}\bi{i-a}{c-a}e_{i-a+1}(x_1,\ldots,x_{i-1})x_{i,\ldots,b}\\
&\hspace{60pt}+(-1)^{c-a}\sum_{i=c}^{b-1}\bi{i-a}{c-a}e_{i-a+1}(x_1,\ldots,x_{i-1})x_{i+1,\ldots,b+1}\\
&\hspace{60pt}+(-1)^{b-a-2}\Big(\sum_{i=c}^{b-1}(-1)^{c-i}\bi{i-a}{c-a}\bi{b-a+1}{i-a+1}-1\Big)e_{b-a+1}(x_1,\ldots,x_{b-1})x_{b}\\
&\hspace{60pt}+(-1)^{b-a-1}\Big(\sum_{i=c}^{b-1}(-1)^{c-i}\bi{i-a}{c-a}\bi{b-a+1}{i-a+1}-1\Big)e_{b-a+1}(x_1,\ldots,x_{b-1})x_{b+1}\\
&\hspace{60pt}+x_{a,\ldots,b+1} .
\end{align*}
Now, by Lemma~\ref{nikou}, the third and the fourth summands are both equal to $-(-1)^{c-b}\bi{b-a}{c-a}\bi{b-a+1}{b-a+1}$ which is in fact $(-1)^{c-b-1}\bi{b-a}{c-a}$. Combining these terms to the first and the second summands as their $b$-th terms, we obtain
\begin{align*}
&\hspace{-20pt}x_c(x_a\cdots x_{b}) \\
&= \ (-1)^{c-a-1}\sum_{i=c}^{b}\bi{i-a}{c-a}e_{i-a+1}(x_1,\ldots,x_{i-1})x_{i,\ldots,b}\\
&\hspace{40pt}+(-1)^{c-a}\sum_{i=c}^{b}\bi{i-a}{c-a}e_{i-a+1}(x_1,\ldots,x_{i-1})x_{i+1,\ldots,b+1} \ + \ x_{a,\ldots,b+1} .
\end{align*}
Thus, the claim holds for the length $l+1$, and the proof completes by induction.
\end{proof}

\vspace{10pt}

\begin{example}\label{exxsq}
{\rm
For $n=7$ and $a=3,c=4,b=5$, we have
\begin{align*}
x_4(x_3x_4x_5) \ &= \ \sum_{i=4}^5\bi{i-3}{1}e_{i-2}(x_1,\ldots,x_{i-1})x_{i,\ldots,5}\\
&\hspace{40pt}-\sum_{i=4}^5\bi{i-3}{1}e_{i-2}(x_1,\ldots,x_{i-1})x_{i+1,\ldots,5}\ +\ x_{3,4,5,6}\\
\ &= \ e_{2}(x_1,\ldots,x_3)x_{4,5}+2e_{3}(x_1,\ldots,x_4)x_{5}\\
&\hspace{14pt}-e_{2}(x_1,\ldots,x_3)x_{5,6}-2e_{4}(x_1,\ldots,x_4)x_{6}+x_{3,4,5,6} .
\end{align*}
This recovers the example which we gave in \eqref{eg10}.
By writing elementary symmetric polynomials explicitly, we obtain
\begin{align*}
x_4(x_3x_4x_5) \ &= \ 2x_{1,2,3,5}+3x_{1,2,4,5}+3x_{1,3,4,5}+3x_{2,3,4,5}\\
&\hspace{14pt}-2x_{1,2,3,6}-2x_{1,2,4,6}-x_{1,2,5,6}-2x_{1,3,4,6}-x_{1,3,5,6}-2x_{2,3,4,6}-x_{2,3,5,6}+x_{3,4,5,6}.
\end{align*}
}
\end{example}

\vspace{10pt}

\begin{remark}\label{remarkcoe}
{\rm As one can see from the computation in Example~\ref{exxsq}, the coefficients of square free monomials of the first summand in \eqref{xsqueq} are not obtained directly from \eqref{xsqueq}. The reason is that the last index $i-1$ in $e_{i-a+1}(x_1,\ldots,x_{i-1})$ and the first index $i$ in $x_{i,\ldots,b}$ are consecutive in \eqref{xsqueq}. This causes that summands having different $i$ may give the same monomial: this was what we saw in Example~\ref{exxsq}.
We will describe the formula that directly gives the coefficient of square free monomials in Section~\ref{opque} (see Proposition~\ref{cor5.2}) with open questions.}
\end{remark}

\vspace{10pt}

As mentioned in the beginning of this section, $x_n$ does not appear in the tautological basis of $H^{*}(\mathcal{Y})$. 
But it does appear in the expansion formula given in Theorem~$\ref{xsqu}$ when $b=n-1$. To express the expansion by the tautological basis, we must exclude $x_n$ by using the relation $x_n=-(x_1+\cdots +x_{n-1})$. We address that this process is not so obvious since it will create squares on the remaining classes $x_1,\ldots,x_{n-1}$. After resolving these new squares, we can obtain the expansion by the tautological basis.
Before stating the formula, let us prepare the next lemma which will be a key step to remove $x_n$.

\begin{lemma}\label{b=n-1 lem}
For $1\le i\le n$, the following holds in $H^{*}(\mathcal{Y})$$:$
\begin{align*}
x_ix_{i+1}\cdots x_n \ = \ (-1)^{n-i+1}\frac{1}{(n-i+1)!}\varpi_{i-1}\varpi_{i}\cdots\varpi_{n-1} .
\end{align*}
\end{lemma}

\begin{proof}
We prove the claim by induction on the length $l\coloneqq n-i+1\,(\ge 1)$ of $x_ix_{i+1}\cdots x_n$. When $l=1$ (i.e.\ when $i=n$),  recall that we have the relation $x_1+\cdots+x_n=0$ given in Corollary~\ref{petx}. This can be written as 
\begin{align*}
x_n \ = \ -\varpi_{n-1}
\end{align*}
by \eqref{xvpiey}. This gives the claim for $l=1$. \\
\indent We next assume that the claim holds for some length $1\le l<n$, and we prove the claim for the length $l+1$. For that purpose, take $1\le i\le n$ such that $n-i+1=l+1$. Then we have $n-(i+1)+1=l$ so that the inductive hypothesis implies that
\begin{align*}
x_{i+1}x_{i+2}\cdots x_n \ = \ (-1)^{n-i}\frac{1}{(n-i)!}\varpi_i\varpi_{i+1}\cdots\varpi_{n-1}.
\end{align*}
Multiplying $x_i$ on both sides, we obtain
\begin{align*}
x_ix_{i+1}\cdots x_n \ = \ (-1)^{n-i}\frac{1}{(n-i)!}x_i\varpi_i\varpi_{i+1}\cdots\varpi_{n-1}.
\end{align*}
Inserting $x_i=\varpi_i-\varpi_{i-1}$ to the right hand side, we obtain
\begin{align*}
&x_ix_{i+1}\cdots x_n\\
\ &\hspace{20pt}= \ (-1)^{n-i}\frac{1}{(n-i)!}(\varpi_i-\varpi_{i-1})\varpi_i\varpi_{i+1}\cdots\varpi_{n-1}\\
\ &\hspace{20pt}= \ (-1)^{n-i}\frac{1}{(n-i)!}(\varpi_i^2\varpi_{i+1}\cdots\varpi_{n-1}-\varpi_{i-1}\varpi_i\cdots\varpi_{n-1})\\
\ &\hspace{20pt}= \ (-1)^{n-i}\frac{1}{(n-i)!}\Big(\frac{n-i}{n-i+1}\varpi_{i-1}\varpi_i\cdots\varpi_{n-1}+\frac{1}{n-i+1}\varpi_i\cdots\varpi_n\\
&\hspace{120pt}-\frac{n-i+1}{n-i+1}\varpi_{i-1}\varpi_i\cdots\varpi_{n-1}\Big)\hspace{20pt}(\text{by Proposition~\ref{varpisqu}})\\
\ &\hspace{20pt}= \ (-1)^{n-i+1}\frac{1}{(n-i+1)!}\varpi_{i-1}\cdots\varpi_{n-1},
\end{align*}
where we used $\varpi_n=0$ for the last equality. Hence, the desired equality holds for $l+1$, and the proof completes by induction.
\end{proof}

\vspace{10pt}

We now state the claim which treats the case $b=n-1$ of Theorem~\ref{xsqu} in terms of tautological basis (i.e.\ square free monomials in $x_1,\ldots,x_{n-1}$).

\begin{theorem}\label{b=n-1}
For $1\le a\le c\le b=n-1$, the following holds in $H^{*}(\mathcal{Y})$$:$
\begin{align*}
x_c(x_a\cdots x_{n-1}) \ = \ (-1)^{c-a-1}\sum_{i=c}^{n}\bi{i-a}{c-a}e_{i-a+1}(x_1,\ldots,x_{i-1})x_{i,\ldots,n-1}, 
\end{align*}
where we mean $x_{i,\ldots,n-1}=1$ when $i=n$.
\end{theorem}

\begin{proof}
By Theorem~\ref{xsqu}, we have
\begin{equation}\label{b=n-13}
\begin{split}
x_c(x_a\cdots x_{n-1}) \ &= \ (-1)^{c-a-1}\sum_{i=c}^{n-1}\bi{i-a}{c-a}e_{i-a+1}(x_1,\ldots,x_{i-1})x_{i,\ldots,n-1}\\
&\hspace{40pt}+(-1)^{c-a}\sum_{i=c}^{n-1}\bi{i-a}{c-a}e_{i-a+1}(x_1,\ldots,x_{i-1})x_{i+1,\ldots,n}\ + \ x_{a,\ldots,n}.
\end{split}
\end{equation}
Let us rewrite the second and the third summands in terms of $x_1,\ldots,x_{n-1}$. Applying Proposition~\ref{piele} and Lemma~\ref{b=n-1 lem} to the second summand, it can be computed as 
\begin{equation}\label{b=n-14}
\begin{split}
&e_{i-a+1}(x_1,\ldots,x_{i-1})x_{i+1,\ldots,n}\\ 
\ &\hspace{20pt}= \ \frac{1}{(i-a+1)!}\varpi_{a-1}\cdots\varpi_{i-1}(-1)^{n-i}\frac{1}{(n-i)!}\varpi_i\cdots\varpi_{n-1}\\
\ &\hspace{20pt}= \ (-1)^{n-i}\frac{(n-a+1)!}{(i-a+1)!(n-i)!}\frac{1}{(n-a+1)!}\varpi_{a-1}\cdots\varpi_{n-1}\\
\ &\hspace{20pt}= \ (-1)^{n-i}\bi{n-a+1}{i-a+1}e_{n-a+1}(x_1,\ldots,x_{n-1}),
\end{split}
\end{equation}
where the last equality follows by Proposition~\ref{piele}. For the third summand  in \eqref{b=n-13}, we have
\begin{equation}\label{b=n-15}
\begin{split}
x_{a,\ldots,n} \ &= \ (-1)^{n-a+1}\frac{1}{(n-a+1)!}\varpi_{a-1}\cdots\varpi_{n-1}\hspace{20pt}(\text{by Lemma~\ref{b=n-1 lem}})\\
\ &= \ (-1)^{n-a+1}e_{n-a+1}(x_1,\ldots,x_{n-1}),
\end{split}
\end{equation}
where the last equality follows by Proposition~\ref{piele}. Applying \eqref{b=n-14} and \eqref{b=n-15} to the second and the third summands in \eqref{b=n-13}, we obtain
\begin{align}\label{b=n-13.5}
x_c(x_a\cdots &x_{n-1})\nonumber\\
\ &= \ (-1)^{c-a-1}\sum_{i=c}^{n-1}\bi{i-a}{c-a}e_{i-a+1}(x_1,\ldots,x_{i-1})x_{i,\ldots,n-1}\nonumber\\
&\hspace{50pt}+(-1)^{c-a}\sum_{i=c}^{n-1}(-1)^{n-i}\bi{i-a}{c-a}\bi{n-a+1}{i-a+1}e_{n-a+1}(x_1,\ldots,x_{n-1})\nonumber\\
&\hspace{50pt}+(-1)^{n-a+1}e_{n-a+1}(x_1,\ldots,x_{n-1})\nonumber\\
\ &= \ (-1)^{c-a-1}\sum_{i=c}^{n-1}\bi{i-a}{c-a}e_{i-a+1}(x_1,\ldots,x_{i-1})x_{i,\ldots,n-1}\nonumber\\
&\hspace{50pt}+(-1)^{n-a}\Bigg(\sum_{i=c}^{n-1}(-1)^{c-i}\bi{i-a}{c-a}\bi{n-a+1}{i-a+1}-1\Bigg)e_{n-a+1}(x_1,\ldots,x_{n-1}).
\end{align}
Now, by Lemma~\ref{nikou}, the coefficient of the second summand in \eqref{b=n-13.5} is equal to $-(-1)^{c-n}\bi{n-a}{c-a}\bi{n-a+1}{n-a+1}$ which is in fact $(-1)^{c-n-1}\bi{n-a}{c-a}$. Combining this term to the first summand in \eqref{b=n-13.5} as its $n$-th term, we obtain
\begin{align*}
x_c(x_a\cdots x_{n-1}) \ = \ (-1)^{c-a-1}\sum_{i=c}^{n}\bi{i-a}{c-a}e_{i-a+1}(x_1,\ldots,x_{i-1})x_{i,\ldots,n-1},
\end{align*}
where we mean $x_{i,\ldots,n-1}=1$ when $i=n$.
\end{proof}

\vspace{10pt}

\begin{example}\label{nsqu}
{\rm
For $n=6$ and $a=3,c=4,b=5$, we have
\begin{align*}
x_4(x_3x_4x_5) \ &= \ \sum_{i=4}^{6}\bi{i-3}{1}e_{i-2}(x_1,\ldots,x_{i-1})x_{i,\ldots,5} \\
\ &= \ e_{2}(x_1,\ldots,x_3)x_{4,5}+2e_{3}(x_1,\ldots,x_4)x_{5}+3e_{4}(x_1,\ldots,x_5)\\[-13pt]\\
\ &= \ 3x_{1,2,3,4}+5x_{1,2,3,5}+6x_{1,2,4,5}+6x_{1,3,4,5}+6x_{2,3,4,5}.
\end{align*}
Having $n=6$ in mind (so that $x_6=-(x_1+\cdots+x_5)$), one may compare this with Example~\ref{exxsq}.
}
\end{example}

\vspace{20pt}

\section{General case}

In this section, we consider expansion formula for possibly non-consecutive monomials which have a single square.
The beginning case is to expand
\begin{align*}
x_lx_a\cdots x_{c-1}x_c^2x_{c+1}\cdots x_b
\end{align*}
for $1\le l<a-1$ and $a\le c\le b\le n-1$.
By applying Theorem~$\ref{xsqu}$ to the part other than $x_l$, we have
\begin{align*}
&x_lx_a\cdots x_{c-1}x_c^2x_{c+1}\cdots x_b \ = \  (-1)^{c-a-1}\sum_{i=c}^{b}\bi{i-a}{c-a}x_le_{i-a+1}(x_1,\ldots,x_{i-1})x_{i,\ldots,b} \\
&\hspace{155pt}+(-1)^{c-a}\sum_{i=c}^{b}\bi{i-a}{c-a}x_le_{i-a+1}(x_1,\ldots,x_{i-1})x_{i+1,\ldots,b+1}\\
&\hspace{155pt}+ \ x_lx_{a,\ldots,b+1} .
\end{align*}
Since the third summand is square free, we focus on the common part
\begin{align}\label{xlk1s0}
x_le_{i-a+1}(x_1,\ldots,x_{i-1})
\end{align}
in the first and the second summand.
As we will see below, this can be expressed as a linear combination of square free monomials in $x_1,\ldots,x_{i-1}$ so that it will not create squares after multiplying the residual terms $x_{i,\ldots,b}$ or $x_{i+1,\ldots,b+1}$.
Noticing that $1\le l<(i-1)-(i-a+1)$ for the indexes in \eqref{xlk1s0}, it suffices to consider the square free expansion of
\begin{align}\label{xlk1s}
x_le_k(x_1,\ldots,x_m) \qquad (1\le l<m-k+1)
\end{align}
for $k\ge0$ and $1\le m\le n-1$.

\vspace{10pt}

The next claim provides the expansion formula for \eqref{xlk1s} by square free monomials in $x_1,\ldots,x_s$.
\begin{proposition}\label{lkas}
Let $k\ge0$ and $1\le m\le n-1$. If $1\le l<m-k+1$, then the following holds in $H^{*}(\mathcal{Y})$$:$ 
\begin{align*}
x_le_k(x_1,\ldots,x_m) \ = \ \sum_{i=0}^ke_i(x_1,\ldots,x_{l+i-1})x_{l+i}e_{k-i}(x_{l+i+1},\ldots,x_m),
\end{align*}
where we are adapting the conventions \eqref{convention of e} and \eqref{convention of 1.5} for elementary symmetric polynomials.
\end{proposition}

\vspace{5pt}

\begin{remark}
{\rm Recall that we have Lemma~\ref{s-k+1ls} which treats \eqref{xlk1s} for the complemental range of $m$. Combining Lemma~\ref{s-k+1ls} with Proposition~\ref{lkas}, we obtain the expansion formula for \eqref{xlk1s} without the restriction of the range of $m$ (note that the range of $k$ is different)}.
\end{remark}

\vspace{5pt}

\begin{proof}[Proof of Proposition~$\ref{lkas}$]
We decompose $e_k(x_1,\ldots,x_m)$ on the left hand side into the part which excludes $x_l$ and the part which contains $x_l$. Then we have
\begin{align}\label{xlek1s}
\begin{split}
&x_le_k(x_1,\ldots,x_m) \\
&\hspace{20pt}\ = \ x_le_k(x_1,\ldots,\check{x_l},\ldots,x_m) \ + \ x_l^2e_{k-1}(x_1,\ldots,\check{x_l},\ldots,x_m) ,
\end{split}
\end{align}
where the symbol $\check{x_l}$ means that we delete $x_l$ from the list. For the elementary symmetric polynomials appearing here, we have
\begin{align*}
e_k(x_1,\ldots,\check{x_l},\ldots,x_m) \ = \ \sum_{i=0}^ke_i(x_1,\ldots,x_{l-1})e_{k-i}(x_{l+1},\ldots,x_m)
\end{align*}
and a similar formula for $e_{k-1}(x_1,\ldots,\check{x_l},\ldots,x_m)$.
Note that we are adapting the convention \eqref{convention of e} and \eqref{convention of 1.5}.
Thus, the equality $(\ref{xlek1s})$ can be written as
\begin{align}\label{xlekasm}
\begin{split}
x_le_k(x_1,\ldots,x_m) 
&\ = \ \sum_{i=0}^ke_i(x_1,\ldots,x_{l-1})x_le_{k-i}(x_{l+1},\ldots,x_m) \\
&\hspace{50pt}\ + \ \sum_{i=0}^{k-1}e_i(x_1,\ldots,x_{l-1})x_l^2e_{k-1-i}(x_{l+1},\ldots,x_m) .
\end{split}
\end{align}
Applying Lemma~$\ref{lem5}$ to the second summand on the right hand side, this computation can be proceeded as
\begin{align}\label{xlekast}
\begin{split}
x_le_k(x_1,\ldots,x_m) 
&\ = \ \sum_{i=0}^ke_i(x_1,\ldots,x_{l-1})x_le_{k-i}(x_{l+1},\ldots,x_m) \\
&\hspace{50pt}\ -\sum_{i=0}^{k-1}e_{i+1}(x_1,\ldots,x_{l-1})x_le_{k-1-i}(x_{l+1},\ldots,x_m)\\
&\hspace{50pt}\ +\sum_{i=0}^{k-1}e_{i+1}(x_1,\ldots,x_l)x_{l+1}e_{k-1-i}(x_{l+1},\ldots,x_m)\\
&\ = \ e_{0}(x_1,\ldots,x_{l-1}) x_le_{k}(x_{l+1},\ldots,x_m) \\
&\hspace{50pt}\ +\sum_{i=0}^{k-1}e_{i+1}(x_1,\ldots,x_l)x_{l+1}e_{k-1-i}(x_{l+1},\ldots,x_m) ,
\end{split}
\end{align}
where the last equality follows since the terms for $i>0$ in the first summand of the previous equality cancel with the terms in the second summand. Note that we left the term $e_{0}(x_1,\ldots,x_{l-1})$ without replacing it with $1$ to make the following computation clearer.
We now decompose $e_{k-1-i}(x_{l+1},\ldots,x_m)$  in the second summand of the last expression into the part which excludes $x_{l+1}$ and the part which contains $x_{l+1}$ (as we did for $x_{l}$ above). Then we have
\begin{align*}
x_le_k(x_1,\ldots,x_m) 
&\ = \ e_{0}(x_1,\ldots,x_{l-1})x_le_{k}(x_{l+1},\ldots,x_m) \\
&\hspace{30pt}\ +\sum_{i=0}^{k-1}e_{i+1}(x_1,\ldots,x_l)x_{l+1}e_{k-1-i}(x_{l+2},\ldots,x_m)\\
&\hspace{30pt}\ +\sum_{i=0}^{k-2}e_{i+1}(x_1,\ldots,x_l)x_{l+1}^2e_{k-2-i}(x_{l+2},\ldots,x_m),
\end{align*}
where we note that the summation for the second summand runs only for $0\le i\le k-2$ since $e_{k-2-i}(x_{l+2},\ldots,x_m)=0$ when $i=k-1$.
To the second and the third summands, we apply the same computation which we did from \eqref{xlekasm} to \eqref{xlekast}. Then we obtain
\begin{align*}
x_le_k(x_1,\ldots,x_m) 
&\ = \ e_{0}(x_1,\ldots,x_{l-1})x_le_{k}(x_{l+1},\ldots,x_m) \\[-13pt]\\
&\hspace{30pt}\ + e_1(x_1,\ldots,x_l)x_{l+1}e_{k-1}(x_{l+2},\ldots,x_m)\\
&\hspace{30pt}\ +\sum_{i=0}^{k-2}e_{i+2}(x_1,\ldots,x_{l+1})x_{l+2}e_{k-2-i}(x_{l+2},\ldots,x_m) .
\end{align*}
Continuing this process, we obtain
\begin{align*}
x_le_k(x_1,\ldots,x_m)
\ &= \ e_{0}(x_1,\ldots,x_{l-1})x_le_k(x_{l+1},\ldots,x_m)\\[-13pt]\\
&\hspace{30pt}+e_1(x_1,\ldots,x_l)x_{l+1}e_{k-1}(x_{l+2},\ldots,x_m)\\[-13pt]\\
&\hspace{30pt}+e_2(x_1,\ldots,x_{l+1})x_{l+2}e_{k-2}(x_{l+3},\ldots,x_m)\\[-13pt]\\
&\hspace{30pt}+\cdots\\[-13pt]\\
&\hspace{30pt}+e_k(x_1,\ldots,x_{l+k-1})x_{l+k}e_0(x_{l+k+1},\ldots,x_m) \\
\ &= \ \sum_{i=0}^ke_i(x_1,\ldots,x_{l+i-1})x_{l+i}e_{k-i}(x_{l+i+1},\ldots,x_m)
\end{align*}
which is the desired equality.
\end{proof}

\vspace{10pt}

Motivated by this proposition, we set
\begin{align*}
e_k^{(l)}(x_1,\ldots,x_m) 
\ \coloneqq \ \sum_{j=0}^{k}e_j(x_1,\ldots,x_{l+j-1})x_{l+j}e_{k-j}(x_{l+j+1},\ldots,x_m)
\end{align*}
for $k\ge0$ and $1\le l\le m\le n-1$. 
We note that the cohomological degree of $e_k^{(l)}(x_1,\ldots,x_m)$ is $2(k+1)$ since $\deg x_i =2$ for $1\le i\le n-1$.
By construction, the above proposition can be stated as follows.

\begin{corollary}\label{lkascor}
Let $k\ge0$ and $1\le m\le n-1$. If $1\le l<m-k+1$, then the following holds in $H^{*}(\mathcal{Y})$$:$
\begin{align*}
x_le_k(x_1,\ldots,x_m) \ = \ e_{k}^{(l)}(x_1,\ldots,x_m).
\end{align*}
\end{corollary}

\vspace{10pt}

Similarly, for $k\ge0$ and $1\le l_1<l_2\le m\le n-1$, we set
\begin{align*}
&e_k^{(l_1,l_2)}(x_1,\ldots,x_m) \\
&\hspace{20pt}
\ = \ 
\sum_{j_2=0}^k\sum_{j_1=0}^{j_2} 
e_{j_{1}}(x_{1},\ldots,x_{d_{1}-1})x_{d_{1}} 
e_{j_{2}-j_1}(x_{d_{1}+1},\ldots,x_{d_{2}-1})x_{d_{2}} 
e_{k-j_2}(x_{d_{2}+1},\ldots,x_m) ,
\end{align*}
where set $d_1\coloneqq l_1+j_1$ and $d_2\coloneqq l_2+j_2$. Note the sum of the indexes of the elementary symmetric polynomials appearing here satisfy $(j_{1})+(j_{2}-j_1)+(k-j_2)=k$. The cohomological degree of $e_k^{(l_1,l_2)}(x_1,\ldots,x_m)$ is $2(k+2).$ \\

More generally, for $k\ge0$ and $1\le l_1<l_2<\cdots <l_s\le m\le n-1$, we set
\begin{align}\label{defofepls}
e_k^{(l_1,\ldots,l_s)}(x_1,\ldots,x_m) 
\ \coloneqq \ 
\sum_{j} \left( \prod_{t=0}^{s-1} e_{j_{t+1}-j_t}(x_{d_{t}+1},\ldots,x_{d_{t+1}-1})x_{d_{t+1}} \right) e_{k-j_s}(x_{d_{s}+1},\ldots,x_m),
\end{align}
where the summation on the right hand side is 
\begin{align*}
\sum_{j} = 
\sum_{j_s=0}^{k} \sum_{j_{s-1}=0}^{j_s} \sum_{j_{s-2}=0}^{j_{s-1}} \cdots  \sum_{j_1=0}^{j_2} , 
\qquad d_t\coloneqq l_t + j_t, \qquad l_0=j_0=0.
\end{align*}
We note that 
\begin{align*}
\deg e_k^{(l_1,\ldots,l_s)}(x_1,\ldots,x_m) = 2(k+s).
\end{align*}
When $s=0$, this agrees with the standard elementary symmetric polynomial:
\begin{align*}
e_k^{\emptyset}(x_1,\ldots,x_m) 
\ = \ 
e_k(x_1,\ldots,x_m).
\end{align*}

\vspace{10pt}

By applying Proposition~$\ref{lkas}$ repeatedly, we obtain the next claim.

\begin{theorem}\label{thm12}
For $1\le l_1<l_2<\cdots<l_s<a-1$, $a\le c\le b\le n-1$ and $M\subseteq\lbrace b+2,\ldots n-1\rbrace$. Then the following holds in $H^{*}(\mathcal{Y})$:
\begin{align*}
x_{l_1,l_2,\ldots,l_s}&x_a\cdots x_{c-1}x_c^2x_{c+1}\cdots x_bx_M\\ 
\ &\hspace{-20pt}= \ (-1)^{c-a-1}\sum_{i=c}^b\bi{i-a}{c-a}e_{i-a+1}^{(l_1,\ldots,l_s)}(x_1,\ldots,x_{i-1})x_{i,\ldots,b}x_M\\
&\hspace{5pt}
+ (-1)^{c-a}\sum_{i=c}^b\bi{i-a}{c-a}e_{i-a+1}^{(l_1,\ldots,l_s)}(x_1,\ldots,x_{i-1})x_{i+1,\ldots,b+1}x_M \ + \ x_{l_1,l_2,\ldots,l_s}x_{a,\ldots,b+1}x_M,
\end{align*}
where $e_{i-a+1}^{(l_1,\ldots,l_s)}(x_1,\ldots,x_{i-1})$ is the polynomial  defined in \eqref{defofepls} of degree $2(i-a+1+s)$.\\
\end{theorem}

\vspace{5pt}

\begin{remark}
{\rm When $b\le n-2$, the class $x_n$ does not appear on the right hand side so that the right hand side is a linear combination of tautological basis of $H^{*}(\mathcal{Y})$.
When $b=n-1$, it does appear on the right hand side. The expansion in terms of tautological basis in this case will be treated in Theorem~\ref{b=n-1gen}.
}
\end{remark}

\vspace{5pt}

\begin{proof}[Proof of Theorem~$\ref{thm12}$]
Since the left hand side and the right hand side have $x_M$ in common, it is enough to prove the equality without $x_M$. Let us write the left hand side of the desired equality (without $x_M$) as
\begin{align*}
x_{l_1,l_2,\ldots,l_s}x_a\cdots x_{c-1}x_c^2x_{c+1}\cdots x_b \ = \ x_{l_1,l_2,\cdots l_{s-1}}(x_{l_s}x_a\cdots x_{c-1}x_c^2x_{c+1}\cdots x_b) .
\end{align*}
Since $1\le l_s<a-1$, the monomial in the parenthesis can be computed by Theorem~\ref{xsqu} and Corollary~\ref{lkascor} as
\begin{align}\label{lsthm10}
\begin{split}
x_{l_s}&x_a\cdots x_{c-1}x_c^2x_{c+1}\cdots x_b\\ 
\ &\hspace{10pt}= \ (-1)^{c-a-1}\sum_{i=c}^b\bi{i-a}{c-a}x_{l_s}e_{i-a+1}(x_1,\ldots,x_{i-1})x_{i,\ldots,b}\\
\ &\hspace{40pt}+ \ (-1)^{c-a}\sum_{i=c}^b\bi{i-a}{c-a}x_{l_s}e_{i-a+1}(x_1,\ldots,x_{i-1})x_{i+1,\ldots,b+1} \ + \ x_{l_s}x_{a,\ldots,b+1} \\
\ &\hspace{10pt}= \ (-1)^{c-a-1}\sum_{i=c}^b\bi{i-a}{c-a}e_{i-a+1}^{(l_s)}(x_1,\ldots,x_{i-1})x_{i,\ldots,b}\\
\ &\hspace{40pt}+ \ (-1)^{c-a}\sum_{i=c}^b\bi{i-a}{c-a}e_{i-a+1}^{(l_s)}(x_1,\ldots,x_{i-1})x_{i+1,\ldots,b+1} \ + \ x_{l_s}x_{a,\ldots,b+1}.
\end{split}
\end{align}
Multiplying $x_{l_{s-1}}$ on both sides, the term $x_{l_{s-1}}e_{i-a+1}^{(l_s)}(x_1,\ldots,x_{i-1})$ appears. For this, we have
\begin{align*}
x_{l_{s-1}}&e_{i-a+1}^{(l_s)}(x_1,\ldots,x_{i-1}) \\
&\ = \ x_{l_{s-1}} \sum_{j_{s}=0}^{i-a+1}e_{j_{s}}(x_1,\ldots,x_{l_{s-1}+j_{s}-1})x_{l_{s}+j_{s}}e_{i-a+1-j_{s}}(x_{l_{s}+j_{s}+1},\ldots,x_{i-1}) 
\end{align*}
by the definition of $e_{i-a+1}^{(l_s)}(x_1,\ldots,x_{i-1})$.
Since $1\le l_{s-1}<(l_s+j_s-1)-j_s+1=l_s$, we can apply Corollary~\ref{lkascor} to the right hand side again. Then this equality becomes
\begin{align*}
x_{l_{s-1}}e_{i-a+1}^{(l_s)}(x_1,\ldots,x_{i-1}) 
&\ = \ \sum_{j_{s}=0}^{i-a+1}e_{j_{s}}^{(l_{s-1})}(x_1,\ldots,x_{l_{s-1}+j_{s}-1})x_{l_{s}+j_{s}}e_{i-a+1-j_{s}}(x_{l_{s}+j_{s}+1},\ldots,x_{i-1}) \\
&\ = \ e_{j_{s}}^{(l_{s-1},l_s)}(x_1,\ldots,x_{i-1}) .
\end{align*}
Hence, after multiplying $x_{l_{s-1}}$ on \eqref{lsthm10}, we obtain
\begin{align*}
x_{l_{s-1}}&x_{l_s}x_a\cdots x_{c-1}x_c^2x_{c+1}\cdots x_b\\ 
\ &= \ (-1)^{c-a-1}\sum_{i=c}^b\bi{i-a}{c-a}e_{i-a+1}^{(l_{s-1},l_s)}(x_1,\ldots,x_{i-1})x_{i,\ldots,b}\\
\ &\hspace{30pt}+ \ (-1)^{c-a}\sum_{i=c}^b\bi{i-a}{c-a}e_{i-a+1}^{(l_{s-1},l_s)}(x_1,\ldots,x_{i-1})x_{i+1,\ldots,b+1} \ + \ x_{l_{s-1}}x_{l_s}x_{a,\ldots,b+1}.
\end{align*}
Continuing this process, we obtain the desired equality.
\end{proof}

\vspace{10pt}

\begin{example}\label{ex4.5}
{\rm For $n=12$, let us express the polinomial
\begin{align}\label{ex4.6}
x_1x_3x_5^2x_6&x_8x_{10}
\end{align}
by square free monomials. Using Theorem~\ref{thm12} to $x_1x_3x_5^2x_6$ in \eqref{ex4.6}, we have
\begin{align*}
x_1x_3x_5^2x_6 \ &= \ -\sum_{i=5}^6e_{i-4}^{(1,3)}(x_1,\ldots,x_{i-1})x_{i,\ldots,6}\\
&\hspace{30pt}
+\sum_{i=5}^6e_{i-4}^{(1,3)}(x_1,\ldots,x_{i-1})x_{i+1,\ldots,7} \ + \ x_{1,3}x_{5,6,7}\\
\ &= \ -e_1^{(1,3)}(x_1,\ldots,x_4)x_{5,6}-e_2^{(1,3)}(x_1,\ldots,x_5)x_6\\
&\hspace{30pt}
+e_1^{(1,3)}(x_1,\ldots,x_4)x_{6,7}+e_2^{(1,3)}(x_1,\ldots,x_5)x_7 \ + \ x_{1,3}x_{5,6,7}.
\end{align*}
Expanding $e_1^{(1,3)}(x_1,\ldots,x_4)$ and $e_2^{(1,3)}(x_1,\ldots,x_5)$ by \eqref{defofepls}, we obtain
\begin{align*}
x_1x_3x_5^2x_6 \ &= \ -3x_{1,2,3,5,6}-6x_{1,2,4,5,6}-5x_{1,3,4,5,6}\\
&\hspace{20pt}+3x_{1,2,3,5,7}+4x_{1,2,4,5,7}+2x_{1,2,4,6,7}+3x_{1,3,4,5,7}+2x_{1,3,4,6,7}+x_{1,3,5,6,7}.
\end{align*}
Substitute the result for \eqref{ex4.5}, we have
\begin{align*}
x_1x_3x_5^2x_6x_8x_{10} \ &= \ -3x_{1,2,3,5,6,8,10}-6x_{1,2,4,5,6,8,10}-5x_{1,3,4,5,6,8,10}\\
&\hspace{20pt}+3x_{1,2,3,5,7,8,10}+4x_{1,2,4,5,7,8,10}+2x_{1,2,4,6,7,8,10}\\
&\hspace{20pt}+3x_{1,3,4,5,7,8,10}+2x_{1,3,4,6,7,8,10}+x_{1,3,5,6,7,8,10}.
\end{align*}
}
\end{example}

\vspace{10pt}

As mentioned in the beginning of Section~\ref{connected case}, the square free monomials in $x_1,\ldots,x_{n}$ are linearly dependent due to the relation $x_1+\cdots+x_n=0$. When $b=n-1$ in the previous theorem, square free monomials containing $x_n$ appears (as in Theorem~\ref{xsqu}).
To obtain the expansion in terms of tautological basis (i.e.\ the square free monomials in $x_1,\ldots,x_{n-1}$), we need to remove $x_n$ from these monomials in square free form. The next claim provides the resulting formula which the reader may compare with  Theorem~\ref{b=n-1}.

\begin{theorem}\label{b=n-1gen}
For $1\le l_1<l_2<\cdots<l_s<a-1<a\le c\le b=n-1$, the following holds in $H^{*}(\mathcal{Y})$:
\begin{align*}
x_{l_1,l_2,\ldots,l_s}&x_a\cdots x_{c-1}x_c^2x_{c+1}\cdots x_{n-1}\\ 
\ &\hspace{-20pt}= \ (-1)^{c-a-1}\sum_{i=c}^{n}\bi{i-a}{c-a}e_{i-a+1}^{(l_1,\ldots,l_s)}(x_1,\ldots,x_{i-1})x_{i,\ldots,n-1} .
\end{align*}
\end{theorem}
\begin{proof}
Since $1\le a\le c\le b=n-1$, the monomial $x_a\cdots x_{c-1}x_c^2x_{c+1}\cdots x_{n-1}$ on the left hand side can be expanded by the tautological basis by Theorem~$\ref{b=n-1}$. Namely, we have
\begin{align*}
x_{l_1,l_2,\ldots,l_s}&x_a\cdots x_{c-1}x_c^2x_{c+1}\cdots x_{n-1}\nonumber\\
\ &= \ x_{l_1,l_2,\ldots,l_s}(-1)^{c-a-1}\sum_{i=c}^n\bi{i-a}{c-a}e_{i-a+1}(x_1,\ldots,x_{i-1})x_{i,\ldots,n-1}\nonumber\\
\ &= \ (-1)^{c-a-1}\sum_{i=c}^n\bi{i-a}{c-a}x_{l_1,l_2,\ldots,l_s}e_{i-a+1}(x_1,\ldots,x_{i-1})x_{i,\ldots,n-1} .
\end{align*}
By the same argument as that used in the previous proof, we have
\begin{align*}
x_{l_1,l_2,\ldots,l_s}e_{i-a+1}(x_1,\ldots,x_{i-1})
= e_{i-a+1}^{(l_1,\ldots,l_s)}(x_1,\ldots,x_{i-1})
\end{align*}
so that we obtain the desired equality.
\end{proof}

\vspace{10pt}

\begin{example}
{\rm For $n=7$ and $l_1=1,l_2=3,a=c=5,b=6$, we obtain
\begin{align*}
x_1x_3x_5^2&x_6\\
\ &= \ -\sum_{i=5}^{7}e_{i-4}^{(1,3)}(x_1,\ldots,x_{i-1})x_{i,\ldots,6}\\
\ &= \ -e_1^{(1,3)}(x_1,\ldots,x_4)x_{5,6}-e_2^{(1,3)}(x_1,\ldots,x_5)x_6-e_3^{(1,3)}(x_1,\ldots,x_6)\\
\ &= \ -4x_{1,2,3,4,6}-9x_{1,2,3,5,6}-12x_{1,2,4,5,6}-9x_{1,3,4,5,6}.
\end{align*}
Having $n=7$ in mind (so that $x_7=-(x_1+\cdots+x_6)$), one may compare this with expand of $x_1x_3x_5^2x_6$ in Example~\ref{ex4.5}.
}
\end{example}

\vspace{10pt}

Recall that we have $x_J=x_{j_1}\cdots x_{j_m}$ for $J=\lbrace j_1<\cdots<j_m\rbrace\subseteq\lbrace 1,\ldots,n-1\rbrace$ and that we have the multiplication rule for the tautological basis:
\begin{align*}
x_J\cdot x_K \ = \ \sum_Lc_{J,K}^Lx_L
\end{align*}
for $J,K\subseteq\lbrace 1,\ldots,n-1\rbrace$.
As we pointed out in Remark~\ref{remarkcoe}, Theorem~\ref{thm12} (and Theorem~\ref{b=n-1gen} as well) provides square free expansion formulas for the case $|J|=1$ with an arbitrary $K$, but the exact coefficients for $x_L$ is not directly obtained from our formula. However, all coefficients of the square free monomials in our expansion formula are integers. Hence, by repeating these formulas, we obtain the following claim.

\begin{corollary}
We have $c_{J,K}^L\in\mathbb{Z}$ for all $J,K,L\subseteq\lbrace 1,\ldots,n-1\rbrace$.
\end{corollary}

\vspace{20pt}

\section{Open questions}\label{opque}

In this section, we list some open questions. 
\begin{question}
How can we write the structure constants $c_{J,K}^L$ explicitly when $|J|=1$?
\end{question}

\vspace{5pt}

Regarding this question, we have the following answer for the consecutive case \eqref{conxeq} (but not for a general case).
\begin{proposition}\label{cor5.2}
For $1\le a\le c\le b\le n-1$, the following holds in $H^{*}(\mathcal{Y})$$:$
\begin{equation}\label{cor5.2eq}
\begin{split}
x_c(x_a\dotsb &x_b)
\ = \ (-1)^{c-a-1}\!\!\sum_{i=a-1}^{b}\Biggl(\bi{b-a+1}{c-a+1}-\bi{i-a}{c-a+1}\Biggl)e_{i-a+1}(x_1,\dotsb,x_{i-2})x_{i,\dotsb,b}\\
&\hspace{60pt}+(-1)^{c-a}\sum_{i=c}^{b}\bi{i-a}{c-a}e_{i-a+1}(x_1,\dotsb,x_{i-1})x_{i+1,\dotsb,b+1}\ +\ x_{a,\dotsb,b+1},
\end{split}
\end{equation}
where we set $\bi{h}{m}=0$ when $h<m$ and we have $e_0(x_1,\dotsb,x_{i-2})=1$ by convention \eqref{convention of e}.
\end{proposition}
\vspace{10pt}

Note that the coefficient of the first summand on the right hand side of \eqref{cor5.2eq} is nonnegative:
\begin{align*}
\bi{b-a+1}{c-a+1}-\bi{i-a}{c-a+1} \ \ge \ 0.
\end{align*}
The difference between \eqref{cor5.2eq} and \eqref{xsqueq} is as follows: in the first summand of \eqref{cor5.2eq}, the last index $i-2$ in $e_{i-a+1}(x_1,\ldots,x_{i-2})$ and the first index $i$ in $x_{i,\ldots,b}$ are not consective. Therefore, the issue in  Remark~\ref{remarkcoe} does not occur in the formula \eqref{cor5.2eq}. That is, it gives us the coefficients for the tautological basis directly.

To prove \eqref{cor5.2eq}, we collect the same monomials, and we do not need to use non-trivial relations in $H^*(\mathcal{Y})$. To keep this paper short, we rather not write down the proof. (It would be interesting to start with giving a proof of \eqref{cor5.2eq}.)
Instead, let us give a convincing example.
Let $n=7$ and $a=3,c=4,b=5$. Then we have
\begin{align*}
x_4(x_3x_4x_5) \ &= \ 
\sum_{i=2}^{5}\Biggl(\bi{3}{2}-\bi{i-3}{2}\Biggl)e_{i-2}(x_1,\dotsb,x_{i-2})x_{i,\dotsb,5}\nonumber\\
&\hspace{60pt}-\sum_{i=4}^{5}\bi{i-3}{2}e_{i-2}(x_1,\dotsb,x_{i-1})x_{i+1,\dotsb,6}+x_{3,4,5,6}\nonumber\\
\ &= \ 3x_{2,3,4,5}+3e_1(x_1)x_{3,4,5}+3e_2(x_1,x_2)x_{4,5}+2e_3(x_1,x_2,x_3)x_5\nonumber\\
&\hspace{14pt}-e_{2}(x_1,\ldots,x_3)x_{5,6}-2e_{4}(x_1,\ldots,x_4)x_{6}+x_{3,4,5,6} .
\end{align*}
One can easily see that there are no overlapping terms, in contrast to \eqref{xsqueq}.
By writing elementary symmetric polynomials explicitly, we obtain
\begin{align*}
x_4(x_3x_4x_5) \ &= \ 2x_{1,2,3,5}+3x_{1,2,4,5}+3x_{1,3,4,5}+3x_{2,3,4,5}\\
&\hspace{14pt}-2x_{1,2,3,6}-2x_{1,2,4,6}-x_{1,2,5,6}-2x_{1,3,4,6}-x_{1,3,5,6}-2x_{2,3,4,6}-x_{2,3,5,6}+x_{3,4,5,6}.
\end{align*}
We address that this results agrees with the computation given in Example~\ref{exxsq}.

\vspace{10pt}
It is also interesting to study the structure constants $c_{J,K}^L$ for the general case (i.e.\ $|J|\ge1$). 
\begin{question}
Can we write down an explicit formula for $c_{J,K}^L$?
\end{question}

The main tool used in this paper is mostly algebraic computations. We now suggest the next question.

\begin{question}
What is the geometric or combinatorial meaning of $c_{J,K}^L$? 
\end{question}

\newpage
\section{Appendix}

\begin{lemma}\label{nikouapp}
For $d>b\ge0$, we have
\begin{align}\label{indhypapp}
\sum_{i=b}^{d-1} (-1)^{i-b}\bi{i}{b}\bi{d}{i+1} \ = \ 1.
\end{align}
\end{lemma}

\begin{proof}
We prove the claim by induction on $d\ge1$.
When $d=1$, we have $b=0$ so that the claim is obvious.

We assume that the claim holds for some $d\ge1$, that is, we assume that \eqref{indhypapp} holds, and we prove the same equality with $d$ replaced by $d+1$. For that, we have
\begin{align*}
\sum_{i=b}^{d}(-1)^{i-b}\bi{i}{b}\bi{d+1}{i+1} 
\ &= \ \sum_{i=b}^{d-1}(-1)^{i-b}\bi{i}{b}\bi{d+1}{i+1}+(-1)^{d-b}\bi{d}{b}\\
\ &= \ \sum_{i=b}^{d-1}(-1)^{i-b}\bi{i}{b}\left\{\bi{d}{i+1}+\bi{d}{i}\right\}+(-1)^{d-b}\bi{d}{b}\\
\ &= \ 1\ +\ \sum_{i=b}^{d-1}(-1)^{i-b}\bi{i}{b}\bi{d}{i}\ +\ (-1)^{d-b}\bi{d}{b}\\
&\hspace{40pt}\text{(by applying \eqref{indhypapp} to the first summand)}\\
\ &= \ 1 \ +\ \sum_{i=b}^{d}(-1)^{i-b}\bi{i}{b}\bi{d}{i} .\\
\end{align*}
Here, the first summand can be computed as
\begin{align*}
\sum_{i=b}^{d}(-1)^{i-b}\bi{i}{b}\bi{d}{i} 
\ &= \ \sum_{i=b}^{d}(-1)^{i-b}\frac{i!}{b!(i-b)!}\frac{d!}{i!(d-i)!}\\
\ &= \ \sum_{i=b}^{d}(-1)^{i-b}\frac{d!}{b!(d-b)!}\frac{(d-b)!}{(i-b)!(d-i)!}\\
\ &= \ \bi{d}{b}\sum_{i=b}^{d}(-1)^{i-b}\bi{d-b}{i-b}\\
\ &= \ \bi{d}{b}\sum_{j=0}^{l}(-1)^j\bi{l}{j} \qquad (l\coloneqq d-b)\\
\ &= 0 .
\end{align*}
Hence, the first summand in the previous equality vanishes, and we obtain the desired equality. So the proof completes by induction.
\end{proof}


\vspace{10pt}

\begin{thebibliography}{9}
\bibitem{b}
H. Abe, T. Horiguchi, H. Kuwata and H. Zeng, \emph{Geometry of Peterson Schubert calculus in type A and left-right diagrams}, Algebr. Comb. \textbf{7} (2024), no. 2, 383--412.
\bibitem{c}
H. Abe and H. Zeng, \emph{The integral cohomology rings of Peterson varieties in type A}, to appear in Acta Mathematica Sinica.
\bibitem{j}
E. Drellich, \emph{Monk's Rule and Giambelli's Formula for Peterson Varieties of All Lie Types}, J. Algebraic Combin. \textbf{41} (2015), no. 2, 539--575.
\bibitem{a}
Y. Fukukawa, M. Harada, and M. Masuda, 
\emph{The equivariant cohomology rings of Peterson varieties}, 
J. Math. Soc. Japan \textbf{67} (2015), no. 3, 1147--1159. 
\bibitem{h}
R. Goldin and B. Gorbutt, \emph{A positive formula for type A Peterson Schubert calculus}, Matematica \textbf{1} (2022), no. 3, 618--665.
\bibitem{i}
R. Goldin, L. Mihalcea and R. Singh, \emph{Positivity of Peterson Schubert Calculus}, Adv. Math. \textbf{455} (2024), Paper No. 109879, 34 pp.
\bibitem{g}
Tao Gui, Yuqi Jia, Xinkai Yu, Zhexi Zhang, Yuchen Zhu, 
\emph{Structure constants of Peterson Schubert calculus},
arXiv:2508.05457.
\bibitem{l}
M. Harada, T. Horiguchi, and M. Masuda, \emph{The equivariant cohomology rings of Peterson varieties in all Lie types}, 
Canad. Math. Bull. \textbf{58} (2015), no. 1, 80--90. 
\bibitem{k}
M. Harada, T. Horiguchi, S. Murai, M. Precup, J. Tymoczko, \emph{A filtration on the cohomology rings of regular nilpotent Hessenberg varieties},
Math. Z. \textbf{298} (2021), no. 3--4, 1345--1382.
\bibitem{m}
T. Horiguchi, \emph{Mixed Eulerian numbers and Peterson Schubert calculus}, Int. Math. Res. Not. \textbf{2024} (2), pp. 1422--1471
\end{thebibliography}
\end{document}